\documentclass[12pt]{amsart}

\usepackage{amsmath}
\usepackage{amssymb}
\usepackage{amsfonts}
\usepackage{amsthm}
\usepackage{enumerate}
\usepackage{hyperref}
\usepackage{graphicx, xcolor}
\usepackage{color}

\theoremstyle{plain}
\newtheorem{thm}{Theorem}[section]

\newtheorem{lem}[thm]{Lemma}

\theoremstyle{definition}

\newtheorem{rem}[thm]{Remark}

\newtheorem{dfns-rems}[thm]{Definitions and Remarks}
\newtheorem{notas-rems}[thm]{Notations and Remarks}
\newtheorem{exmps-rems}[thm]{Examples and Remarks}

\begin{document}

% ------------------------------------------------------------------------

\title{Nowhere-zero 5-flow on signed ladders}

% ------------------------------------------------------------------------
\author[L. Parsaei-Majd]{Leila Parsaei-Majd}

\address{L. Parsaei-Majd, Hasso Plattner Institute, University of Potsdam, Germany}

\email{leila.parsaei84@yahoo.com}

% ------------------------------------------------------------------------

\begin{abstract}
	In 1983, Bouchet conjectured that every flow-admissible signed graph admits a nowhere-zero $6$-flow. In this paper, we prove that Bouchet's conjecture holds for all signed ladders, circular and M\"obius ladders. In fact, all signed ladders, circular and M\"obius ladders admit a nowhere-zero $5$-flow except for one case of signed circular ladders. Of course, the exception also has a nowhere-zero $6$-flow. 
	%Finally, we show that the signed generalized Petersen graphs $P(n,2)$ for any $n\geqslant 6$, under certain conditions, have a nowhere-zero $5$-flow. 
\end{abstract}

% ------------------------------------------------------------------------


% ------------------------------------------------------------------------


% ------------------------------------------------------------------------

% ------------------------------------------------------------------------

\maketitle

%%%%%%%%%%%%%%%%%%%%%%%%%%%%%%%%%%%%%%%%%%%%%%%%%%%%%%%%%%%%%%%%%%%%%%%%%%
\section{Introduction}
%We can find many signed graphs with no nowhere zero5-flow in Nowhere-zero flows in signed graphs: A survey, witten by Kaiser et al. 
A \textit{signed graph} is a graph with each edge labelled with a sign, $+$ or 
$-$. An assignment of signs to each edge is a signature.
An orientation of a signed graph is obtained by dividing each edge
into two half-edges each of which receives its own direction. A positive edge
has one half-edge directed from and the other half-edge directed to its 
end-vertex. Hence, a negative edge has both half-edges directed either towards 
or from their respective end-vertices. 

Let $v$ be a vertex of a signed graph $G$. \textit{Vertex switching} at $v$ changes the sign of each edge incident with $v$ to its
opposite. Let $X\subseteq V$. 
Switching a vertex set $X$ means reversing the signs of all edges between $X$
and its complement. Switching a set $X$ has the
same effect as switching all the vertices in $X$, one after another.\\
Two signed graphs $G$ and $G^{\prime}$ with the same underlying graph 
but possibly different signatures on their edges are \textit{switching 
	equivalent}, if there is a series of switchings that transforms $G$ to 
$G^{\prime}$. Switching equivalence is an equivalence relation on the 
signatures of a fixed graph.
If $G^{\prime}$ is isomorphic to a switching of $G$, we say that $G$ and 
$G^{\prime}$ are \textit{switching isomorphic}.\\
% and we write $G\simeq G^{\prime}$. 
A signed graph is \textit{balanced} if and only if it is switching equivalent to the signed graph
with all-positive signature. And a signed graph is \textit{anti-balanced} if it is switching equivalent to
the signed graph with all-negative signature. In other words, a signed graph is balanced
(anti-balanced) if and only if every circuit of the underlying graph contains
an even (odd) number of negative edges, as vertex switching preserves the
parity of the number of negative edges around a circuit.\\

Signed graphs were introduced by Harary \cite{harary} as a model for social networks. Also, signed graphs have diverse applications, a more recent one is in quantum computing \cite{applications}. We refer to \cite{{harary},{zaslavsky1}} for more information about signed graphs.

In this paper, we use the symbols $CL_{n}$ and $ML_{n}$ to denote a signed 
circular ladder and a signed M\"obius ladder of order $2n$, respectively. A \textit{nowhere-zero $k$-flow} on a signed graph $G$ is an assignment of an orientation and a value from $\{\pm 1, \pm 2, \ldots, \pm (k - 1)\}$ to each 
edge in such a way that for each vertex of $G$ the sum of incoming values equals 
the sum of outgoing values (Kirchhoff's law). We call such graphs flow-admissible.
In 1983, Bouchet in \cite{[4]} conjectured that every flow-admissible signed graph admits a nowhere-zero $6$-flow. For example, Bouchet showed that there is a signature for the Petersen
graph which admits no nowhere-zero $5$-flow. Edita M\'{a}\v{c}ajov\'{a} found a case of $CL_4$, Fig. \ref{cl4}, using a computer search which has no nowhere-zero $5$-flow but admits a nowhere-zero $6$-flow. For further examples of signed graphs with this properties refer to \cite{survey}.\\
However, signed circular ladders do not produce any more examples: in
Sections 2 and 3 we show that all signed circular ladders apart from Fig. \ref{cl4} and all signed M\"obius ladders admit a nowhere-zero $5$-flow.\\

\section{Nowhere-zero flow on signed circular ladder $CL_{n}$}

\begin{thm}[{\cite[Theorem 4.4]{twonegative}}]\label{twoneg}
	Let $G$ be a flow-admissible signed cubic graph with two negative 
	edges. If $G$ is bipartite, then it has a nowhere-zero $k$-flow with $k\leqslant 4$.
\end{thm}

We have the following lemma due to K\"{o}nig \cite{konig}.

\begin{lem}
	Every $r$-regular bipartite graph, $r \geqslant 1$, is $1$-factorable.
\end{lem}

Therefore, we can decompose each cubic bipartite graph into $1$-factors.
Consider a signed graph $G$ carrying a $k$-flow $\phi$ and let $P = e_1e_2 
\ldots e_r$ be
an $u-v$ trail in $G$. By sending a value $b \in \{\pm 1, \pm 2, \ldots, \pm (k 
- 1)\}$ from $u$ to $v$ along $P$ we mean
reversing the orientation of the edge $e_1$ so that it leaves $u$, adding $b$ to
$\phi (e_1)$, and adding $\pm b$ to $\phi (e_i)$ for all other edges of $P$ in 
such a way that Kirchhoff's law is fulfilled at each inner vertex of $P$.

\begin{thm}\label{2-factors}
Let $G$ be a signed cubic bipartite graph and $\{F_1, F_2, F_3\}$ be a 
$1$-factorisation of $G$. Consider $2$-factors
$F_1\cup F_2$, $F_1\cup F_3$ and $F_2\cup F_3$. 
If two of them are balanced, then $G$ admits a nowhere-zero $4$-flow.
\end{thm}
\begin{proof}
Without loss of generality, assume that $F_1\cup F_2$ and $F_2\cup F_3$ are 
balanced $2$-factors. Nowhere-zero $4$-flow is obtained by sending value $1$ 
along each circuit of $F_1\cup F_2$ and value $2$
along $F_2\cup F_3$.
\end{proof}

\begin{rem}\label{numbnega}
Note that considering switching equivalence the maximum number of negative edges 
can occur in the subladder given in Fig. \ref{cub8}, is $3$.

\begin{figure}
[!htb]
\minipage{0.55\textwidth}
\includegraphics[width=\linewidth]{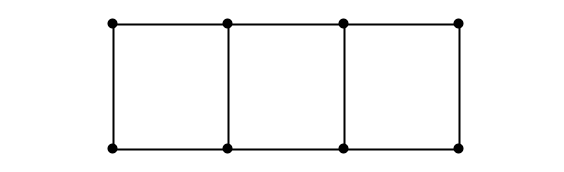}
\caption{}\label{cub8}
\endminipage
\end{figure}

In fact, just in three cases three negative edges occur, and for other cases by 
switching at some vertices the number of negative edges declines. The possible 
cases are listed in Fig. \ref{twocases}, in which dashed lines denote negative edges.
\end{rem}
\begin{figure}
[!htb]
\minipage{0.85\textwidth}
\includegraphics[width=\linewidth]{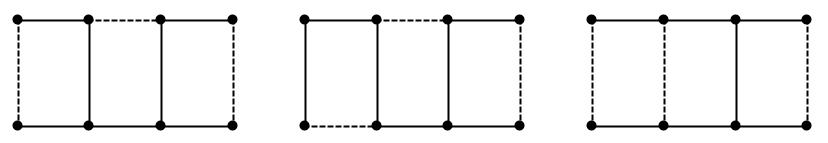}
\caption{}\label{twocases}
\endminipage
\end{figure}

\begin{rem}\label{numbernegativec}
One can check that in a signed circular ladder $CL_{n}$, up to switching 
equivalence there exist at most $[\frac{n}{2}]+1$ negative edges. 
\end{rem}

\begin{lem}\label{cl5,6}
Signed circular ladders $CL_{5}$ and $CL_{6}$ with any signature have 
a nowhere-zero $5$-flow.
\end{lem}

\begin{proof}
By Remark \ref{numbernegativec}, a signed circular ladder $CL_{5}$ has at most 
three negative edges. All types of the signed graph $CL_{5}$ with two or three 
negative edges are listed in Fig. \ref{cl5}.

\begin{figure}[!htb]
\minipage{0.95\textwidth}
\includegraphics[width=\linewidth]{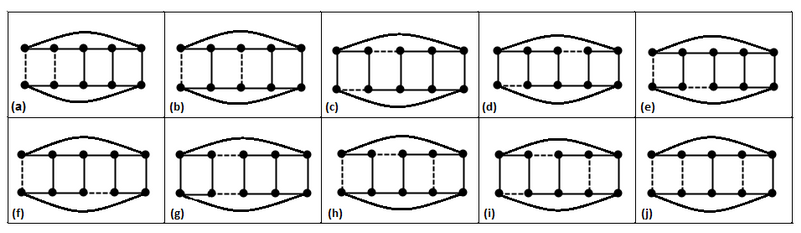}
\caption{$CL_{5}$ with two and three negative edges}\label{cl5}
\endminipage
\end{figure}

It is not hard to check that all signed graphs $CL_{5}$ with two negative edges 
given in Fig. \ref{cl5}, have a nowhere-zero $4$-flow, and signed graphs 
$CL_{5}$ with three negative edges, Graphs (h), (i), (j) in Fig. \ref{cl5}, 
admit a nowhere-zero $5$-flow.

Now, we show that the signed graph $CL_{6}$ with any signature (with at least 
two negative edges) has a nowhere-zero $5$-flow. 
By Theorem \ref{twoneg}, if the signed graph $CL_{6}$ has two negative edges, it 
admits a nowhere-zero $4$-flow.
Moreover, all types of the signed graph $CL_{6}$ with three and four negative 
edges are listed in Fig. \ref{cl6}. All of them have a nowhere-zero $5$-flow 
(Graphs (b), (f), (g), and (h) have nowhere-zero $4$-flow, and Graphs (a), (c), 
(d), and (e) admit nowhere-zero $5$-flow).

\begin{figure}[!htb]
\minipage{\textwidth}
\includegraphics[width=\linewidth]{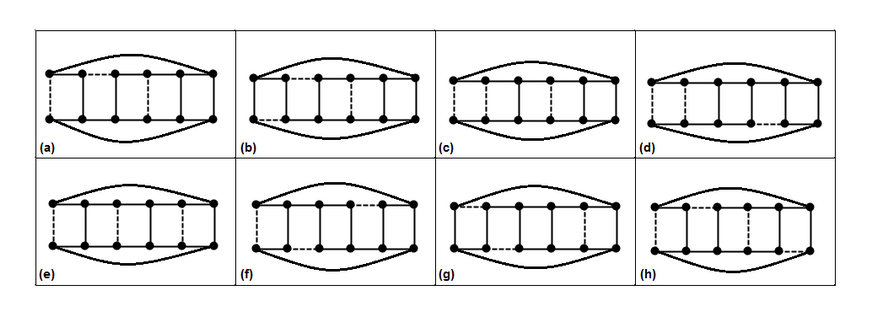}
\caption{$CL_{6}$ with three and four negative edges}\label{cl6}
\endminipage
\end{figure}
 
By Remark \ref{numbernegativec}, up to switching equivalence there is no 
$CL_{6}$ with more than four negative edges.
So, we conclude that the signed graph $CL_{6}$ with two or three negative edges 
has nowhere-zero $5$-flow.
\end{proof}

\begin{figure}[!htb]
\minipage{0.45\textwidth}
\includegraphics[width=\linewidth]{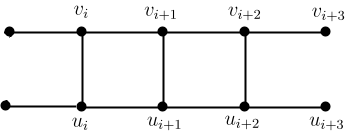}
\caption{A subladder $H$}\label{positiverangs}
\endminipage
\end{figure}

\begin{lem}\label{deleterangs}
Let $(CL_{n}, \sigma)$ be a flow admissible signed circular ladder with 
$n\geqslant 7$ having a positive square $S$ (in Fig.~\ref{positiverangs}, 
$S=v_{i+1}v_{i+2}u_{i+2}u_{i+1}v_{i+1}$.) 
If $((CL_{n}\setminus S)\cup\{v_{i}v_{i+3}, u_{i}u_{i+3}\}, \sigma)$ with 
$$\sigma(v_{i}v_{i+3})=\sigma( v_{i}v_{i+1})\sigma(v_{i+2}v_{i+3}) 
\quad \text{and}\quad 
\sigma(u_{i}u_{i+3})=\sigma(u_{i}u_{i+1})\sigma(u_{i+2}u_{i+3})$$ 
has a nowhere-zero $5$-flow, then $(CL_{n}, \sigma)$ has also a nowhere-zero 
$5$-flow.
\end{lem}
\begin{proof}
Up to switching equivalence there are four cases for the edges of $S$ and  the 
edge set $\{v_{i}v_{i+1}, v_{i+2}v_{i+3}, u_{i}u_{i+1}, u_{i+2}u_{i+3}\}$:

\begin{itemize}

\item[(i)]

If all edges of $S\cup \{v_{i}v_{i+1}, v_{i+2}v_{i+3}, u_{i}u_{i+1}, 
u_{i+2}u_{i+3}\}$ are positive.

\item[(ii)]

If $v_{i+1}u_{i+1}$ and $v_{i+2}u_{i+2}$ are negative.

\item[(iii)]

If $u_{i}u_{i+1}$ and $v_{i+2}v_{i+3}$ are negative.

\item[(iv)]

If just an edge $u_{i}u_{i+1}$ is negative.

\end{itemize}
 
Consider a nowhere-zero $5$-flow on $(CL_{n}\setminus S)\cup\{v_{i}v_{i+3}, 
u_{i}u_{i+3}\}$ with $\sigma(v_{i}v_{i+3})=\sigma( 
v_{i}v_{i+1})\sigma(v_{i+2}v_{i+3})$ and 
$\sigma(u_{i}u_{i+3})=\sigma(u_{i}u_{i+1})\sigma(u_{i+2}u_{i+3})$. We show that 
$(CL_{n}, \sigma)$ in each of the above cases has a nowhere-zero $5$-flow. It is 
sufficient to prove the assertion for one of the cases, the rest cases are 
proved similarly.
Let all edges of $S\cup \{v_{i}v_{i+1}, v_{i+2}v_{i+3}, u_{i}u_{i+1}, 
u_{i+2}u_{i+3}\}$ are positive. Without loss of generality we can assume that 
$v_{i}u_{i}$ is positive. Consider the left signed graph in Fig. \ref{NZF}, 
which $a, b, b^{\prime}, c, c^{\prime}\in \{\pm 1, \pm 2, \pm 3, \pm 4\}$. It is 
not hard to check that one can find a value $x\in \{\pm 1, \pm 2, \pm 3, \pm 
4\}$ such that $x+c, x+c^{\prime}\in \{\pm 1, \pm 2, \pm 3, \pm 4\}$, see the 
right signed graph in Fig. \ref{NZF}.

\begin{figure}[!htb]
\minipage{\textwidth}
\includegraphics[width=\linewidth]{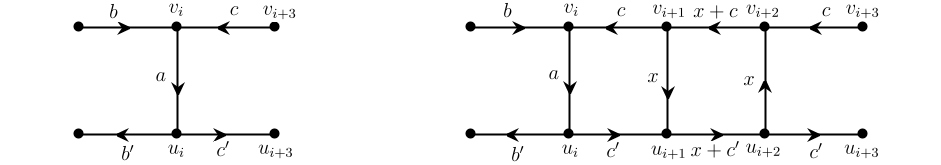}
\caption{Constructing a nowhere-zero $5$-flow on $CL_{n}$}\label{NZF}
\endminipage
\end{figure}

A note about other three cases, for example assume that $u_{i}u_{i+1}$ and 
$v_{i+2}v_{i+3}$ are negative. Delete $S$ and consider a nowhere-zero $5$-flow 
on the obtained signed circular ladder $CL_{n-2}$, see Fig. \ref{case2}. Similar 
to the mentioned case, and using this note that if $c=c^{\prime}=4$, then we 
achieve a contradiction because $(CL_{n}, \sigma)$ for $n\geqslant 7$, is flow 
admissible. (Since we assume that $(CL_{n}, \sigma)$ for $n\geqslant 7$, is flow 
admissible, there is a positive integer $\ell$ such that $(CL_{n}, \sigma)$ 
admits a nowhere-zero $\ell$-flow. Considering the right subladder in Fig. 
\ref{case2}, if $c=c^{\prime}=\ell -1$, we achieve a contradiction with flow 
admissibility of $(CL_{n}, \sigma)$. Now, we claim that $(CL_{n}, \sigma)$ has a 
nowhere-zero $5$-flow, so we can ignore this equality $c=c^{\prime}=4$).

\begin{figure}[!htb]
\minipage{\textwidth}
\includegraphics[width=\linewidth]{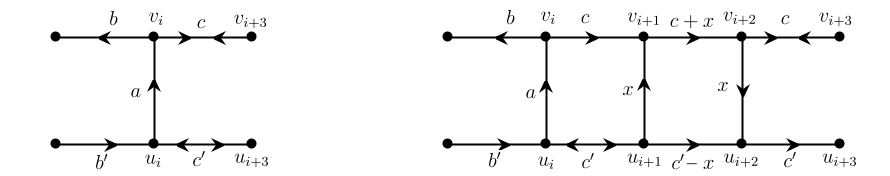}
\caption{Constructing a nowhere-zero $5$-flow on $CL_{n}$}\label{case2}
\endminipage
\end{figure}

\end{proof}

In the following theorem, we prove that the signed graph $CL_{n}$ for 
$n\geqslant 5$, with any signature admits a nowhere-zero $5$-flow. 

\begin{thm}
Let $n\geqslant 5$ be a positive integer. If a signed circular ladder $CL_{n}$ 
is flow admissible, then it has a nowhere-zero $5$-flow.
\end{thm}

\begin{proof}
If $n=5$ or $6$, then by Lemma \ref{cl5,6}, signed circular ladders $CL_{5}$ and 
$CL_{6}$ have nowhere-zero $5$-flow.
Assume that $k\geqslant 3$ is a positive integer which shows the number of 
negative edges in $CL_n$ for $n\geqslant 7$.
It is not hard to check that if signed graph $CL_{2k}$ has $k$ negative edges 
(with any signature), it has at least a positive square except in one case which 
all negative edges occur on rungs alternately. This exceptional case admits a 
nowhere-zero $4$-flow if $k$ is even, see Theorem \ref{2-factors}, and for odd 
$k$, one can find a pattern using Graphs (a) and (b) in Fig. \ref{le}, which 
shows that it has a nowhere-zero $5$-flow. Also, if $CL_{2k}$ has $k+1$ negative 
edges, there is just one signature with all negative squares; it is given in 
Fig. \ref{except2k}, three points among the rungs means there is a positive rung 
and then a negative rung, alternately. It has a nowhere-zero $5$-flow, see 
Graphs (c) and (d) in Fig. \ref{le}. Note that if $k$ is odd, then we can 
conclude that the exceptional case of $CL_{2k}$ with $k+1$ negative edges, given 
in Fig. \ref{except2k}, has nowhere-zero $4$-flow because it has two balanced 
$2$-factors, see Theorem \ref{2-factors}.

\begin{figure}[!htb]
\minipage{0.87\textwidth}
\includegraphics[width=\linewidth]{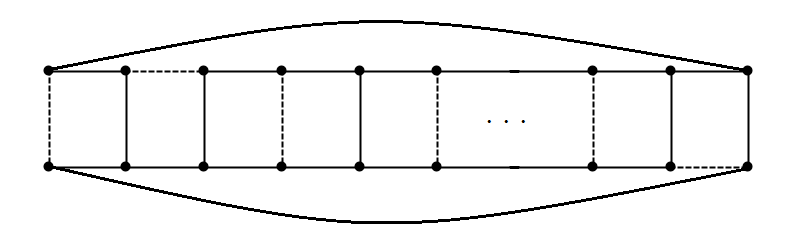}
\caption{The exceptional case of $CL_{2k}$ with $k+1$ negative 
edges}\label{except2k}
\endminipage
\end{figure}

Similarly, signed graph $CL_{2k+1}$ with $k$ or $k+1$ negative edges, with any 
signature, has at least a positive square except in one case with $k+1$ negative 
edges given in Fig. \ref{except2k+1}. One can find a certain pattern to exist a 
nowhere-zero $5$-flow on the exceptional case of $CL_{2k+1}$ with $k+1$ negative 
edges, see Graphs (e), (f), (g), and (h) in Fig. \ref{le}. 
Now, ignore the three exceptional cases of signed graphs $CL_{2k}$ and 
$CL_{2k+1}$ with $k$ and $k+1$ negative edges. 

\begin{figure}[!htb]
\minipage{0.90\textwidth}
\includegraphics[width=\linewidth]{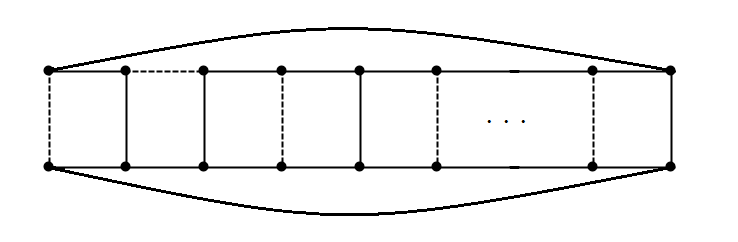}
\caption{The exceptional case of $CL_{2k+1}$ with $k+1$ negative 
edges}\label{except2k+1}
\endminipage
\end{figure}

We claim that the signed graphs $CL_{2k}$ and $CL_{2k+1}$ with $k$ and $k+1$ 
negative edges have nowhere-zero $5$-flow. Note that $k+1$ is the maximum number 
of the negative edges can occur in $CL_{2k}$ and $CL_{2k+1}$.
We prove the claim by induction on $k\geqslant 3$. Let $k=3$. We know that 
$CL_{6}$ with three or four negative edges has nowhere-zero $5$-flow. Also, we 
know that signed graph $CL_{7}$ with three and four negative edges (except the 
exceptional case given in Fig. \ref{except2k+1}) has at least a positive square. 
So, by Lemma \ref{deleterangs} and the existence of a nowhere-zero $5$-flow on 
$CL_5$, we conclude that $CL_{7}$ with three and four negative edges has 
nowhere-zero $5$-flow.
Now, assume that for each $4\leqslant i \leqslant k-1$, $CL_{2i}$ and 
$CL_{2i+1}$ with $i$ and $i+1$ negative edges have nowhere-zero $5$-flow. This 
is the induction hypothesis.
Consider $i=k$. Since $CL_{2k}$ and $CL_{2k+1}$ with $k$ or $k+1$ negative edges 
have at least one positive square (without considering exceptional cases). 
Hence, by Lemma \ref{deleterangs} and using the induction hypothesis, we 
conclude that each signed graph $CL_{n}$ for $n\geqslant 5$, with at least two 
negative edges has a nowhere-zero $5$-flow.

\end{proof}

\begin{lem}
The signed graph $CL_{4}$ has a nowhere-zero $6$-flow.
\end{lem}
\begin{proof}
It is sufficient to check the signed circular ladders $CL_{4}$ with two and 
three negative edges.
If $CL_{4}$ has two negative edges by Theorem \ref{twoneg}, it has a 
nowhere-zero $4$-flow.
Also, there is just one signed circular ladder $CL_{4}$ with three negative 
edges (up to switching equivalence). It is given in Fig. \ref{cl4}, and it has a 
nowhere-zero $6$-flow. Note that it does not have a nowhere-zero $k$-flow for 
some positive integer $k<6$.
\end{proof}
\begin{figure}[!htb]
\minipage{0.60\textwidth}
\includegraphics[width=\linewidth]{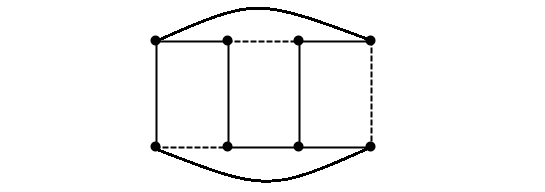}
\caption{}\label{cl4}
\endminipage
\end{figure}

%%%%%%%%%%%%%%%%%%%%%%%%%%%%%%%%%%%%%%%%%%%%%%%%%%%%%%%%%%%%%
\section{Nowhere-zero flow on signed M\"obius ladder $ML_{n}$}

In this section, we are going to show that signed M\"obius ladders $ML_{n}$ have a nowhere-zero $5$-flow.

\begin{rem}\label{maxnegm}
Note that in a signed M\"obius ladder $ML_{n}$, up to switching equivalence 
there exist at most $[\frac{n+1}{2}]$ negative edges. Moreover, M\"obius ladders 
$ML_{n}$ for odd $n$, are bipartite.
\end{rem}

So by Remark \ref{maxnegm}, $ML_{4}$ and $ML_{5}$ have at most $2$ and $3$ 
negative edges, respectively. 

\begin{lem}\label{ml4,5}
The signed M\"obius ladders $ML_{4}$ and $ML_{5}$ admit a nowhere-zero $5$-flow.
\end{lem}
\begin{proof}
All types of signed M\"obius ladder $ML_{4}$ with two negative edges (up to 
switching isomorphic) are listed in Fig. \ref{ml4}. It is not hard to find a 
nowhere-zero $4$-flow on Graphs (a), (b), (c), and (e), and Graph (d) has a 
nowhere-zero $5$-flow.
\begin{figure}[!htb]
\minipage{1\textwidth}
\includegraphics[width=\linewidth]{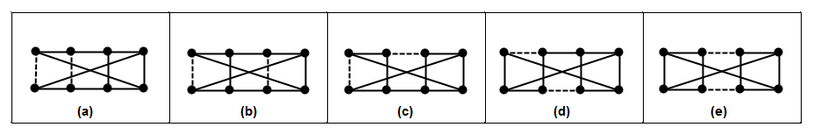}
\caption{All types of $ML_{4}$ with two negative edges}\label{ml4}
\endminipage
\end{figure}

All types of signed M\"obius ladder $ML_{5}$ with two negative edges have a 
nowhere-zero $4$-flow, see Theorem \ref{twoneg}. 
Now, if signed M\"obius ladder $ML_{5}$ has three negative edges, up to 
switching isomorphic there are just two types, see Fig. \ref{ml5}. Graph (a) and 
Graph (b) given in Fig. \ref{ml5}, have nowhere-zero $5$-flow and nowhere-zero 
$4$-flow, respectively.
\end{proof}
\vspace{-0.2 cm}
\begin{figure}[!htb]
\minipage{0.70\textwidth}
\includegraphics[width=\linewidth]{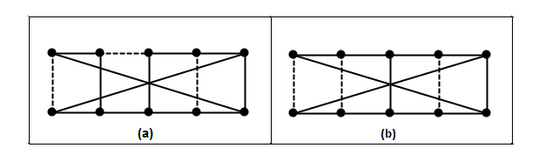}
\caption{$ML_{5}$ with three negative edges}\label{ml5}
\endminipage
\end{figure}

We can state a lemma similar to Lemma \ref{deleterangs} for signed M\"obius 
ladders.

\begin{lem}\label{deleterangsm}
Let the signed M\"obius ladder $(ML_{n}, \sigma)$ for $n\geqslant 6$, be flow 
admissible and have a positive square, $S$. In Fig. \ref{positiverangs}, let 
$S=v_{i+1}v_{i+2}u_{i+2}u_{i+1}v_{i+1}$. 
If $((ML_{n}\setminus S)\cup\{v_{i}v_{i+3}, u_{i}u_{i+3}\}, \sigma)$ with 
$$\sigma(v_{i}v_{i+3})=\sigma( v_{i}v_{i+1})\sigma(v_{i+2}v_{i+3})\quad 
\text{and}\quad\sigma(u_{i}u_{i+3})=\sigma(u_{i}u_{i+1})\sigma(u_{i+2}u_{i+3}
)$$ has a nowhere-zero $5$-flow, then $(ML_{n}, \sigma)$ has also a 
nowhere-zero $5$-flow.
\end{lem}
\begin{proof}
Proceed as in the proof of Lemma \ref{deleterangs}.
\end{proof}

\begin{thm}
Let $n\geqslant 4$ be a positive integer. If a signed M\"obius ladder $ML_{n}$ 
is flow admissible, then it has a nowhere-zero $5$-flow.
\end{thm}

\begin{proof}
If $n=4$ or $5$, then by Lemma \ref{ml4,5}, signed M\"obius ladders $ML_{4}$ and 
$ML_{5}$ with any signature have nowhere-zero $5$-flow.
Let $k\geqslant 3$ is a positive integer which denotes the number of negative 
edges in $ML_{n}$ for $n\geqslant 6$.
One can check that if signed graph $ML_{2k}$ has $k-1$ or $k$ negative edges 
(with any signature), it has at least a positive square except in one case which 
all negative edges occur on rungs alternately. This exceptional case admits a 
nowhere-zero $4$-flow. Using nowhere-zero $4$-flow on Graphs (a), (b), (c), and 
(d) in Fig. \ref{el}, we can obtain a pattern for the existence a nowhere-zero 
$4$-flow on this exceptional case.
Moreover, signed M\"obius ladders $ML_{2k+1}$ with $k$ or $k+1$ negative edges, 
with any signature, have at least a positive square except in one case with 
$k+1$ negative edges given in Fig. \ref{exceptm2k+1}. This exceptional case has 
nowhere-zero $5$-flow, see Graphs (e), (f), (g), and (h) in Fig. \ref{el}. Now, 
ignore the three exceptional cases of signed graphs $ML_{2k}$ with $k-1$ and $k$ 
negative edges and $ML_{2k+1}$ with $k$ and $k+1$ negative edges. 

\begin{figure}[!htb]
\minipage{0.85\textwidth}
\includegraphics[width=\linewidth]{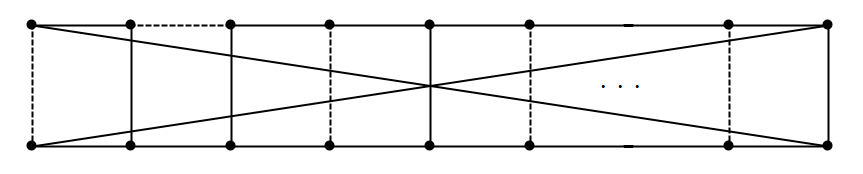}
\caption{The exceptional case of $ML_{2k+1}$ with $k+1$ negative edges}\label{exceptm2k+1}
\endminipage
\end{figure}

In order to prove the assertion, it is sufficient to show that signed graphs $ML_{2k}$ with $k$ negative edges and $ML_{2k+1}$ with $k+1$ negative edges have nowhere-zero $5$-flow. Note that $k$ and $k+1$ are the maximum number of the negative edges can occur in $ML_{2k}$ and $ML_{2k+1}$, respectively.
To prove the claim use induction on $k\geqslant 3$. Let $k=3$. One can check that $ML_{6}$ with three negative edges has at least a positive square, (as it mentioned just in one case there is not a positive square). Also, signed graph $ML_{7}$ with four negative edges (except the exceptional case given in Fig. \ref{exceptm2k+1}) has at least a positive square. So, by Lemma \ref{deleterangs} and the existence of a nowhere-zero $5$-flow on $ML_{4}$ and $ML_{5}$, we conclude that $ML_{6}$ and $ML_{7}$ with three and four negative edges, respectively have nowhere-zero $5$-flow.
Now, assume that for each $4\leqslant i \leqslant k-1$, $ML_{2i}$ and $ML_{2i+1}$ with $i$ and $i+1$ negative edges, respectively have nowhere-zero $5$-flow. This is the induction hypothesis.
Consider $i=k$. Since $ML_{2k}$ and $ML_{2k+1}$ with $k$ and $k+1$ negative edges, respectively have at least one positive square (without considering exceptional cases). Hence, by Lemma \ref{deleterangs} and using the induction hypothesis, we conclude that each signed graph $ML_{n}$ for $n\geqslant 4$, with at least two negative edges has a nowhere-zero $5$-flow.
\end{proof}

\begin{rem}
Note that the next natural class to look at are the generalized Petersen graphs, see \cite{GPG1}. The reason of considering this family of graphs is that the Petersen graph is exceptional among all generalized Petersen graphs by not admitting a nowhere-zero $4$-flow. Since the proof by Castagna and Prins in \cite{GPG2} that shows that all other generalized Petersen graphs admit a nowhere-zero $4$-flow (are
$3$-edge-colourable) is far from easy, and there is no other shorter proof, it is plausible, that establishing a nowhere-zero $5$-flow for signed generalized Petersen graphs is not going to be easy.
\end{rem}

\begin{figure}[!htb]
\minipage{0.90\textwidth}
\includegraphics[width=\linewidth]{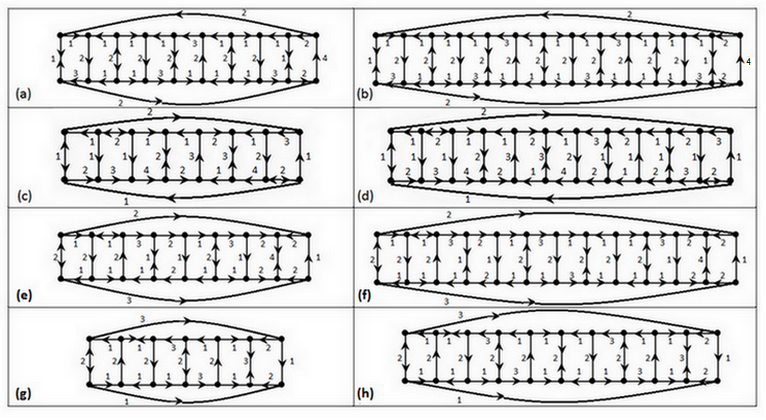}
\caption{Nowhere-zero $5$-flow on the exceptional cases of $CL_{n}$}\label{le}
\endminipage
\end{figure}

\begin{figure}[!htb]
\minipage{0.90\textwidth}
\includegraphics[width=\linewidth]{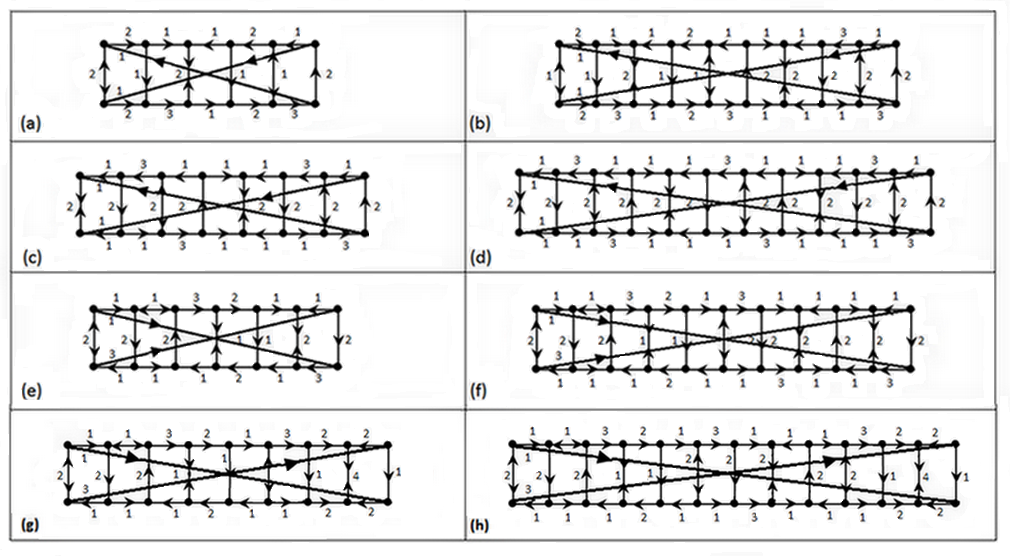}
\caption{Nowhere-zero $5$-flow on the exceptional cases of $ML_{n}$}\label{el}
\endminipage
\end{figure}
%%%%%%%%%%%%%%%%%%%%%%%%%%%%%%%%%%%
%%%%%%%%%%%%%%%%%%%%%%%%%%%%%%%%%%%%%%%%
%\section*{Acknowledgment}
%We would like to thank Martin \v{S}koviera and Edita M\'{a}\v{c}ajov\'{a} for suggesting to work on nowhere-zero flow for signed ladders and for their helpful comments.
%%%%%%%%%%%%%%%%%%%%%%%%%%%%%%%%%%%%%%%%%%%%%%%%%%%%%

\end{document}